\documentclass{article}
\usepackage[utf8]{inputenc}
\usepackage[T2A]{fontenc}
\usepackage{amssymb}
\usepackage{amsmath}
\usepackage{amsthm}
\usepackage[mathscr]{eucal}
\usepackage{graphics}
\usepackage[dvips]{graphicx}
\usepackage{epsfig}

\usepackage{indentfirst,float,cite}
\usepackage{graphicx}

\title{Asymptotics of the Number of Endpoints of a Random Walk on a Certain Class of Directed Metric Graphs}

\author{V.\,L.~Chernyshev,
National Research University Higher School of Economics (HSE),\\
Moscow, Russian Federation,
vchern@mech.math.msu.su\\
A.\,A.~Tolchennikov,
A. Ishlinsky Institute for Problems in Mechanics RAS, \\
M.V. Lomonosov Moscow State University, \\
Moscow, Russian Federation,
tolchennikovaa@gmail.com}

\newcommand{\nc}{\newcommand}
\nc{\e}{\varepsilon}
\nc{\al}{\alpha}
\nc{\be}{\beta}
\nc{\ga}{\gamma}
\nc{\la}{\lambda}
\nc{\ph}{\varphi}
\nc{\pa}{\partial}

\theoremstyle{definition}

\newtheorem{theorem}{Theorem}

\numberwithin{equation}{section}

\begin{document}

\maketitle

\begin{abstract}
A certain class of directed metric graphs is considered. Asymptotics for a number of possible endpoints of a random walk at large times is found.
\end{abstract}

\section{Introduction}

Let us consider a directed metric graph. Let us denote length of edge $e_k$ by $t_k$ and suppose that all lengths $\{ t_k \}_{k=1}^E$ are linearly independent over the set of rational numbers $\mathbb{Q}$.

One could consider a random walk (see \cite{RW}) on a directed metric graph (see, for example, \cite{KB} for references on metric graphs). The main unlikeness with the often considered case (see, for example, \cite{Lovasz}) is that the endpoint of a walk can be any point on an edge of a metric graph, and not only one of the vertices. Let one point start its move along the graph from a vertex (a source) at the initial moment of time. The passage time for each individual edge is fixed. In each vertex, the point with some non-zero probability selects one of the outgoing edges for further movement. Backward turns on the edges are prohibited in this model. Our aim is to analyze an asymptotics of the number $N(T)$ of possible endpoints of such random walk as time $T$ increases. The only assumption about the probabilities of choosing an edge is that it is non-zero for all edges, i.e. a situation of a general position. Such random walk could naturally arise while studying the dynamical systems on various networks.

The asymptotics for finite compact metric non-directed graphs was constructed in papers \cite{Chernyshev2017, RCD}. Moreover, in paper \cite{ChSh} the problem for wave propagation on singular manifolds was reduced to considering graphs with a finite number of vertices, but with infinite valences (but only a finite number of edges are involved at any finite time). So this article is the first to discuss the asymptotics of possible end-points on digraphs.

We will consider a finite directed graph $G=(V,E)$
of the following form: let there be an outgoing tree with vertices $v$ and a root
$s$, oriented in the direction from the root. The graph $G$ is obtained from this tree
by adding a finite number of edges leading from some vertices
$s$ to $s$. These graphs have an important property: there is only one route (that does not return to the root) from the root to any vertex.

\textbf{Definition} We call a  directed strongly connected metric graph a \textit{one-way Sperner} graph (see \cite{Salii} for details about Sperner graphs) if it consists of one-way tree started from the \textit{source} vertex $S$ and has a backward edges only leading to the source.

\begin{figure}[ht]
  \includegraphics[width=5cm]{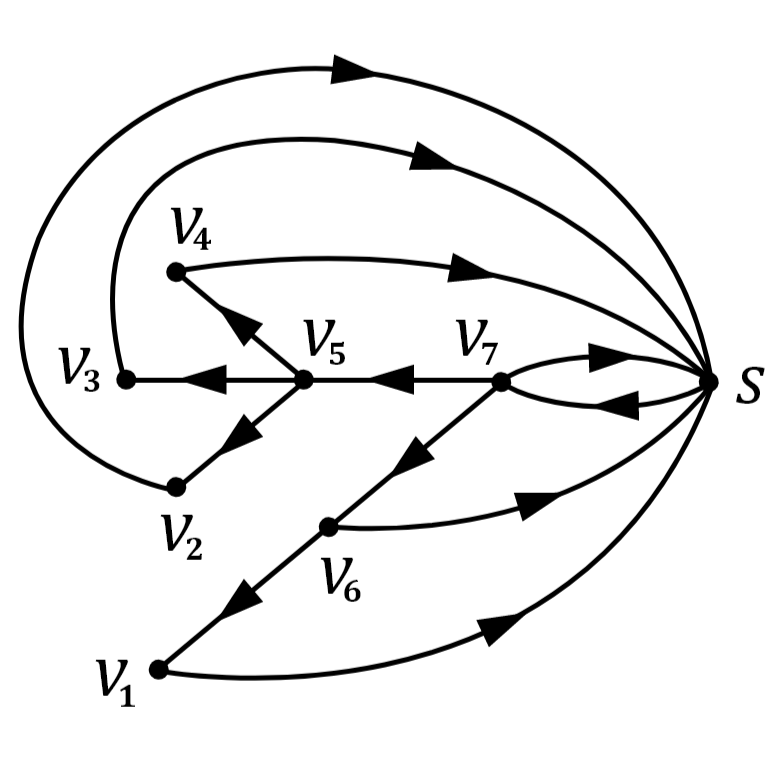}
  \caption{Example of strong digraph of discussed type.}
\end{figure}

Thus, in what follows we will consider only one-way Sperner graphs.

\section{Exact combinatorial formula for $N(T)$}

Let us introduce several notations.

For any subgraph $G'$ of the graph $G$ and the vertex $v$ we denote by
$\rho_{in}{(G',v)}$ and $\rho_{out}(G',v)$ the number of incoming edges incoming to $v$ and the number of outgoing from $v$ edges of the subgraph $G'$.

For any route $\mu$ we denote by $t(\mu)$ the time of this route, i.e. the sum of the times of passage of the edges included in $\mu$.

In our graph $G$ for any vertex $v\in V$ there is a unique simple chain $l_v$ (i.e., a route in which all vertices are pairwise distinct) from
$s$ to $v$ (in particular, $l_s = \emptyset$).

Let $c_1,\ldots,c_k$ be the elementary cycles of the graph $G$ (in our case $k = \rho_{in}(G,s) $). The set of all elementary cycles we denote by $C = \{c_1,\ldots, c_k\}$.

In what follows we will assume that the passage times of all elementary cycles and all simple chains of the form $l_v$ in the aggregate are linearly independent over $\mathbb{Q}$. This is a situation of general position.

Let us consider a linear inequality of the form $n_1 a_1 + \ldots + n_j a_j \le T$, where $n_1, \ldots, n_j \in \mathbb{N}$. By $\# \{n_1 a_1 + \ldots + n_j a_j \le T\}$ we will denote the number of natural solutions of this inequality.

\begin{theorem}
  \begin{equation}  N(T)=\sum_{v\in V} \sum_{I\subset
	  \{1,\ldots,k\} } \left( \rho_{out}(G, v) -
		\rho_{in}(G', v)
	    \right) \# \left\{ t(l_v) + \sum_{i\in I} n_i t(c_i) \leq T
	    \right\},
		\label{exact}
	  \end{equation}
  where the subgraph $G'=G'(v,I)$ of the graph $G$ is formed by the union of the edges
  of the simple chain $l_v$ and elementary cycles $\{c_i\}_{i\in I}$.

\end{theorem}

\begin{proof}
Let us consider an arbitrary route $\mu$ from $s$ to $v$. This route can be represented in the form of passing along the edges of several elementary cycles and along the edges of a simple chain $l_v$.
In that case the passage time $t(\mu)$ of the route $\mu$ has the form $t(l_v)+ \sum_{i\in I} n_i t(c_i) $, where $n_i\in \mathbb{N}$. Note that any time this form is the time of passage of a route from the $s$ to $v$. But, of course, different routes can have identical transit times. So, we have described the set $M$ of passage times for routes starting at the vertex $s$. It has the form:
$$
M= \sqcup_{v\in V} \sqcup_{I\subset \{1,\ldots,k\} } M_{v,I},
$$
where
$$
M_{v,I} =
\{ t(l_v) +
  \sum_{i\in I} n_i t(c_i)\ | \ n_i \in \mathbb{N} \ \forall i \in I
\}
$$
are the passage times for routes ending at the vertex $v$, and passing along all edges of the cycles $c_i (i\in I)$. The condition of linear independence over $\mathbb{Q}$ of the passage times
of elementary chains and elementary cycles ensure that the union is disjoint.

Now note that the function $N(T)$ is piecewise constant and jumps can occur only during the times $M$ of passing routes. The jump occurs at time $t\in M_{v,I}$ is equal to $\rho_{out}(G,v) - \rho_{in}(G',v)$ (where the subgraph $G'$ is formed by the edges of an elementary chain $l_v$ and cycles $c_i\ (i\in I)$), since for any $t\in M_{v,I}$ there is a route $\mu$ with $t(\mu)=t$, which ends along any given edge $G'$, entering the vertex $v$.

It remains to sum up the jumps over all passage times not exceeding $T$ and obtain the value of the function $N(T)$.

\end{proof}

\section{Asymptotic formula for $N(T)$ as $t\to \infty$}

\begin{theorem}
Let $G$ be a finite one-way Sperner metric graph. We consider a random walk on it with the initial vertex $s$. Then for the number of possible endpoints at the time $T$ has the following asymptotics:
\[
N(T)=\frac{T^{\beta -1}}{(\beta -1)!} \cdot \frac{\sum\limits_{e\in E}
t(e)}{\prod\limits_{i=1}^{\beta} t(c_j)}(1+o(T^{\beta-1})),
\]
where $\beta$ is a number of elementary cycles in $\Gamma$ and $T$ tends to infinity.
\end{theorem}
\begin{proof}
  We know (see \cite{RCD} for references on Barnes---Bernoilli polynomials) that the number of non-negative solutions to the inequality $n_1 a_1
... + n_m a_m\le T$ grows as a polynomial of degree $m$. Accordingly, to find the leading coefficient, we need to take inequalities in the formula
\eqref{exact}, in which either all cycles ($|I|=\beta$), are involved, or all cycles except one ($|I|=\beta-1$).

Let us consider the term $N_1(T)$ in the formula \eqref{exact} corresponding to $|I|=\beta$ (then
$G'=G$):
$$
N_1(T) = \sum\limits_{v\in V}\left[\rho_{out}(G, v) - \rho_{in}(G,
v)\right] \# \left\{t(l_v) + \sum_{i=1}^{\beta} n_i t(c_i)  \le T \right\}$$

Let us write the two leading terms in the expansion
$\# \left\{t(l_v) + \sum_{i=1}^{\beta} n_i t(c_i)  \le T \right\}$:

$$\# \left\{ \sum_{i=1}^{\beta} n_i t(c_i)  \leq \underbrace{T -
t(l_v)}_{\lambda} \right\} = \frac{1}{\prod_{i=1}^{\beta} t(c_i) } \Bigg(
\frac{\lambda^{\beta}}{\beta!} - \frac{1}{2} \sum_{i=1}^{\beta} t(c_i)
\frac{\lambda^{\beta-1}}{(\beta-1)!} + O(\lambda^{\beta-2}) \Bigg) =
$$
$$= \frac{1}{\beta! \prod_{i=1}^{\beta} t(c_i)} \left[ T^{\beta} - \beta\, \left( t(l_v)
  + \frac12 \sum_{i=1}^{\beta} t(c_i)
  \right)
  T^{\beta-1}
+ O(T^{\beta-2}) \right]
$$
Note that the coefficient at $T^{\beta}$  does not depend on $v$, therefore
$[T^{\beta}] N_1(T) = 0$ due to the fact that
$
\sum\limits_{v\in V}\left[\rho_{out}(G, v) - \rho_{in}(G,
v)\right] = 0,
$
i.e. by hand-shaking lemma.
Thus:
$$
[T^{\beta-1}]N_1(T) = - \frac{1}{(\beta-1)! \prod_{i=1}^{\beta} t(c_i)}
\sum\limits_{v\in V}\left[\rho_{out}(G, v) - \rho_{in}(G,
v)\right] t(l_v)
$$

Now consider the term $N_2(T)$ in the formula \eqref{exact},
corresponding to $I=\{1,\ldots,\beta\}\setminus \{ j \}$:
$$
N_2(T) = \sum_{v\in V}  \sum_{j=1,\ldots, \beta} \left( \rho_{out}(G, v) -
		\rho_{in}(G', v)
	  	    \right) \# \left\{ t(l_v) + \sum\limits_{i  \ne j} n_i
			t(c_i) \leq T
		  	    \right\},
			  $$
			  In this sum $\rho_{in}(G',v) =\rho_{in}(G,v)=1$ for
			  $v\ne s$ and
			  $\rho_{in}(G',s) =\rho_{in}(G,s)-1$.
			
			We get that $N_2(T) = N_2'(T) + N_2''(T)$, where
		  $$
		  N_2'(T) = \sum_{v\in V}   \left( \rho_{out}(G, v) -
				\rho_{in}(G, v)
			  	    \right)
				  		\left(
						\sum_{j=1}^{\beta}
					  		\# \left\{  t(l_v) + \sum\limits_{i  \ne
							j} n_i t(c_i) \leq T
						  	    \right\}
							  	    \right)	
								  		,
									  $$
									$$
								  N_2''(T) =
								  \sum_{j=1}^{\beta}
										\# \left\{   \sum\limits_{i
										\ne j} n_i t(c_i) \leq T
									  	    \right\}
										  $$
										  We have: $[T^{\beta-1}] N_2'(T) =
										0$, therefore
									  $$
									  [T^{\beta-1}] N_2(T) =
									  \frac{\sum_{i=1}^{\beta}
									t(c_i)}{(\beta-1)! \prod_{i=1}^{\beta}
								  t(c_i)}
								$$
							
							  It remains to prove that
							$$
							\sum_{e\in E} t(e) = \sum_{i=1}^{\beta} t(c_{\beta})
						  -
						  \sum\limits_{v\in V}\left[\rho_{out}(G, v) -
							\rho_{in}(G,
						  v)\right] t(l_v)
						$$
					
					  Let us show that by removing the edges, the expressions on the left and right decrease by the same amount.
				First, we remove the edge $e=(v,s)$. The expression on the left decreased by
			   $t(e)$, the expression on the right decreased by $t(l) + t(e) -
			 t(l)$, where
		    the simple chain $l$ leads from $s$ to $v$.
		  We continue until $\beta$ becomes equal to zero, i.e.
		$G$ will become a directed tree.
%

		We remove the hanging edge $e$:
	  the expression on the left decreased by $t(e)$, and the expression on the right decreased by $-t(l)+
  t(l)+t(e)$, where $l$ --- is a simple chain from $s$ to the beginning of the edge $e$.

In the end, we get a tree with one edge and this equality is true for it.
\end{proof}

\subsection{Discussion and Examples}

Earlier (see \cite{ChSh}) a formula was obtained for the leading asymptotic coefficient of $N(T)$ in the case of an ordinary (undirected) graph.
Here is the formula:

$$
N(T)=\frac{T^{E-1}}{2^{V-2}(E-1)!} \cdot \frac{\sum\limits_{i=1}^E q_j}{\prod\limits_{i=1}^E q_j}(1+o(T^{E-1})),
$$
where $E$ is the number of edges in the graph, $V$ is the number of vertices, and $q_j$ is the lengths of the edges of the undirected graph.\\

There is no direct analogue of considered by us class of graphs in the undirected case. But we can consider its special case, namely the class of graphs, which are oriented disjoint `` loops '' (cycles with two vertices), connected at the source. Such a graph corresponds to an undirected star graph. In this case, we can assume that the sum of the lengths of the edges of the directed graph ${\sum\limits_{i=1}^E t_j}$ is equal to $2{\sum\limits_{i=1}^E q_j}$, and $t(c_j)=2q_j$. The number of vertices $V$ is equal to $E+1$, then, substituting into the theorem proved in this article, we get:

$$
N(T)=\frac{T^{E-1}}{(E-1)!} \cdot \frac{2\sum\limits_{i=1}^E q_j}{2^E\prod\limits_{i=1}^E q_j}(1+o(T^{E-1})),
$$
which, after shortening, gives the formula for the undirected case.

We could notice that the presented formula for the leading term of $N(t)$ is valid not only for graphs from the class of digraphs that we considered.

\begin{figure}[ht]
  \includegraphics[width=5cm]{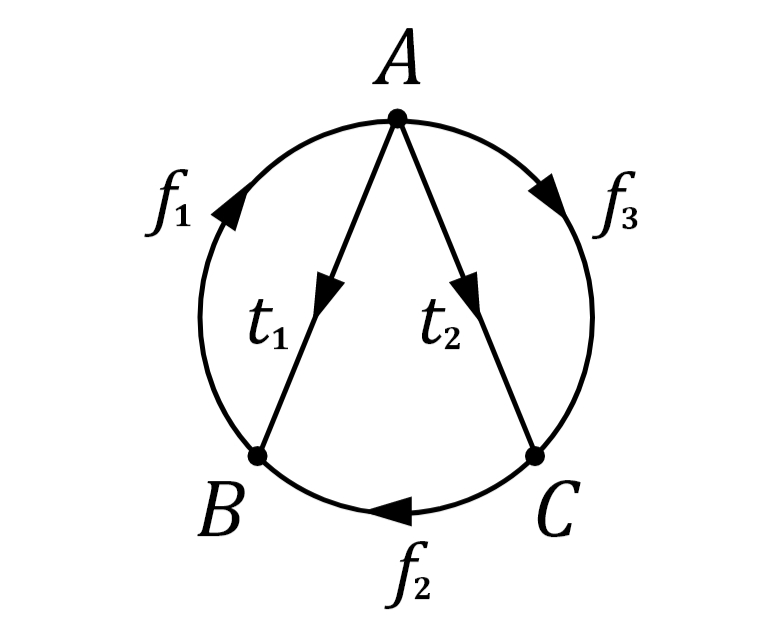}
\end{figure}
Let us consider a graph in the form of a circle with two points on it and two directed chords.\\

The starting vertex is the vertex $A=s$. Let us find the times of possible routes.\\

\begin{enumerate}
  \item Times of routes from $A$ to $A$:
  $$n_1(t_1 + f_1) + n_2(t_2 + f_2 + f_1) + n_3(f_3 + f_2 + f1), \quad n_1, n_2, n_3 \geq 0$$
  At these times $N(T)$ increases by 2.

  \item Times of routes from $A$ to $B$:
  $$n_1(t_1 + f_1) + n_2(t_2 + f_2 + f_1) + n_3(f_3 + f_2 + f_1) - f_1,$$ where not all $n_i$ are zero.\\
  The condition under which at these times the routes end at $t_1$ and $f_2$:
  $$
  \begin{cases}
  n_1 > 0 \\
  \left[
  \begin{array}{c}
       n_2 > 0\\
       n_3 > 0\\
  \end{array}
  \right.
  \end{cases}
  $$
  Here $N(T)$ decreases by 1.

  \item Times of routes from $A$ to $C$:
  $$n_1(t_1 + f_1) + n_2(t_2 + f_2 + f_1) + n_3(f_3 + f_2 + f_1) - f_1 - f_2,$$ where $n_2$ or $n_3 \neq 0$.\\
  The condition under which at these times the routes end at $t_2$ and $f_3$:
  $$
  \begin{cases}
  n_2 > 0 \\
  n_3 > 0 \\
  \end{cases}
  $$
\end{enumerate}

Thus, we get that $N(T)$ is a quadratic function of $T$:

$$N(T) = 2 \#
\left\{
    n_1(t_1 + f_1) + n_2(t_2 + f_2 + f_1) + n_3(f_3 + f_2 + f_1) \leq T, \quad n_1, n_2, n_3 \geq 0
\right\}
-
$$
$$
- \#
\left\{
    n_1(t_1 + f_1) + n_2(t_2 + f_2 + f_1) + n_3(f_3 + f_2 + f_1) - f_1 \leq T, \quad n_1 > 0,
    \left[
    \begin{array}{c}
       n_2 > 0\\
       n_3 > 0\\
    \end{array}
    \right.
\right\}
-
$$
$$
- \#
\left\{
    n_1(t_1 + f_1) + n_2(t_2 + f_2 + f_1) + n_3(f_3 + f_2 + f_1) - f_3 - f_2 \leq T, \quad n_1 \geq 0, n_2 > 0, n_3 > 0
\right\}
=
$$
$$
= T^2 \cdot \frac{f_1 + f_2 + f_3 + t_1 + t_2}{2(f_1 + f_2 + f_3)(f_1 + t_1)(f_1 + f_2 + t_2)} + \underline{\underline{O}}(T)
$$

\subsection*{Conclusions} We found the asymptotics for the number of possible endpoints of a random walk at large times in the case of the certain class of directed graph. Examples show that the formula for the leading coefficient still holds for digraphs without the uniqueness of the path from source to every vertex. So to find a class of metric graph for which the derived formula could be correct could be the aim of further research.

\subsection*{Acknowledgments}
The authors are grateful to V.\,E.~Nazaikinskii,  A.\,I.~Shafarevich for support and useful discussions.
The work of A. Tolchennikov on Section 2 was supported by grant 16-11-10069 of the Russian Science Foundation. The work of V. Chernyshev on Section 3 was supported by the RFBR grant 20-07-01103 a.

\bibliographystyle{plain}

\end{document}